\documentclass[sn-mathphys]{sn-jnl}
\usepackage[all,pdf,2cell]{xy}\UseAllTwocells\SilentMatrices
\usepackage{amsthm}
\usepackage{amsmath}
\usepackage{amsfonts}
\usepackage{enumerate}
\usepackage{cleveref}
\usepackage{graphicx}
\newtheorem{theorem}{Theorem}[section]
\newtheorem{lemma}[theorem]{Lemma}
\newtheorem{fact}[theorem]{Fact}
\newtheorem{proposition}[theorem]{Proposition}
\newtheorem{corollary}[theorem]{Corollary}
\theoremstyle{definition}
\newtheorem{definition}[theorem]{Definition}
\newtheorem{example}[theorem]{Example}
\newtheorem{conjecture}{Conjecture}
\newcommand{\newcategory}[1]{\expandafter\newcommand\csname #1\endcsname{\mathbf{#1}}}
\renewcommand{\Q}{Q^+}
\newcommand{\pp}{{\perp\perp}}
\newcommand{\ud}{{\uparrow\downarrow}}
\newcommand{\du}{{\downarrow\uparrow}}
\newcommand{\compl}[1]{\widehat{#1}}
\newcommand{\cl}{\mathrm{cl}}
\newcommand{\pow}{\mathcal{P}}
\newcommand*{\vcenteredhbox}[1]{\begingroup
\setbox0=\hbox{#1}\parbox{\wd0}{\box0}\endgroup}
\begin{document}
\title{Orthogonality spaces associated with posets}
\author{\fnm{Gejza} \sur{Jenča}}\email{gejza.jenca@stuba.sk}
\affil{\orgdiv{Department of Mathematics and Descriptive Geometry, Faculty of Civil Engineering}, \orgname{Slovak University of Technology},
\orgaddress{\street{Radlinského 11},
\city{Bratislava},
\postcode{810 05},
\country{Slovakia}}}

\abstract{
An orthogonality space is a set equipped with a symmetric, irreflexive relation
called orthogonality. Every orthogonality space has an associated complete
ortholattice, called the logic of the orthogonality space. To every poset, we
associate an orthogonality space consisting of proper quotients (that means,
nonsingleton closed intervals), equipped with a certain orthogonality relation.
We prove that a finite bounded poset is a lattice if and only if the logic
of its orthogonality space is an orthomodular lattice. We prove that that
a poset is a chain if and only if the logic of the associated orthogonality space
is a Boolean algebra.
}

\maketitle

\section{Introduction}

In his PhD. thesis \cite{dacey1968orthomodular}, Dacey explored the notion of
``abstract orthogonality'', by means of sets equipped with a symmetric,
irreflexive relation $\perp$. He named these structures {\em orthogonality spaces}. Every
orthogonality space has an orthocomplementation operator $X\mapsto X^\perp$
defined on the set of all its subsets. Dacey proved that $X\mapsto X^\pp$ is a
closure operator and that the set of all closed subsets of an orthogonality space forms
a complete ortholattice, which we call {\em the logic of an orthogonality space}.
Moreover, Dacey gave a characterization of orthogonality spaces such that their logic is
an orthomodular lattice. The orthogonality spaces of this type are nowadays called {\em Dacey
spaces}.

Let us remark that (as a special case of a {\em polarity between two sets}) the
idea of an orthogonality space and its associated lattice of closed sets
appears already in the classical monograph \cite[Section
V.7]{birkhoff1948lattice}.  Nevertheless, Dacey was probably the first person
that explored the orthomodular law in this context.

Since orthogonality spaces are nothing but (undirected, simple, loopless)
graphs, it is perhaps not surprising that the $X\mapsto X^\perp$ mapping appears
under the name {\em neighbourhood operator} in graph theory. Implicitly, the logic
associated with the neighbourhood operator was used in the seminal paper
\cite{lovasz1978kneser}, in which Lovász proved the Kneser's conjecture.
Explicitly, they were used by Walker in \cite{walker1983graphs}.
In that paper, the ideas from the \cite{lovasz1978kneser} paper were generalized and reformulated in the language of category theory.

For an overview of the results on orthogonality spaces and more general
{\em test spaces}, see \cite{wilce2011test}. See also \cite{paseka2022normal,paseka2022categories} for some more recent results
on orthogonality spaces.

In this paper, we construct an orthogonality space $(\Q(P),\perp)$ 
from every poset $P$. The elements of $\Q(P)$ are pairs $(a,b)$ of elements of $P$
with $a<b$ and the orthogonality relation can be described as ``to be above/below each other
in some chain'' (see \Cref{fig:orthogonality}). Then we examine how are the properties
of $P$ reflected in the structure of the orthogonality space $(\Q(P),\perp)$ and its logic. In the last section,
we describe some straightforward connections between the logic of $\Q(P)$ and two
constructions: the Kalmbach construction and the MacNeille completion.

\section{Preliminaries}

\subsection{Posets, lattices}

A poset $P$ is {\em lower bounded} if it has the smallest element, which we denote by $0$. Similarly, the greatest element of an {\em upper bounded} poset is denoted by $1$. A poset that is both upper and lower bounded is called {\em bounded}.

If $X$ is a subset of a poset $P$, then $a$ is a {\em lower bound of $X$} if and only if for all
$x\in X$, $a\leq x$.  For every set $X\subseteq P$, we write $X^\downarrow$ for
the set of all lower bounds of $X$. The notions of an {\em upper bound}
and $X^\uparrow$ are defined dually. For a singleton subset $\{x\}$, we abbreviate $\{x\}^\downarrow=x^\downarrow$ and $\{x\}^\uparrow=x^\uparrow$.
Note that $\emptyset^\uparrow=\emptyset^\downarrow=P$.

A poset $P$ is a {\em complete lattice} if and only if for all subsets $X$ of $P$,
$X^\uparrow$ has the smallest element, denoted by $\bigvee X$ and called {\em
the smallest upper bound of $X$}. In every complete lattice $P$, 
every subset $X$ of $P$ has a {\em greatest lower bound}, denoted by $\bigwedge X$. We say that a poset $P$ is {\em a lattice} if and only if for all $x,y\in P$, $\bigwedge\{x,y\}$ and $\bigvee\{x,y\}$ exist. In a lattice, we write $x\wedge y=\bigwedge\{x,y\}$ and $x\vee y=\bigwedge\{x,y\}$. A bounded lattice $P$ is {\em complemented} if for every
$x\in P$ there is a $x'\in P$ such that $x\wedge x'=0$ and $x\vee x'=1$. The element $x'$ is then called {\em a complement of $x$}.

A subset $I$ of a poset $P$ is an {\em order ideal of $P$} if for every $x\in I$, $x^\downarrow\subseteq I$. The
dual notion is {\em order filter}.
An order ideal $I$ of a poset $P$ is {\em closed} if and only if $I=I^{\ud}$.
It is well known \cite{macneille1937partially} that every poset $P$ embeds into a complete lattice, called
the {\em Dedekind-MacNeille completion of $P$}, denoted by $\compl P$. The
poset $\compl P$ consists of all closed order ideals of $P$, ordered by
inclusion. There is a canonical embedding $\eta_P\colon P\to\compl P$, given by
the rule $\eta_P(x)=x^\downarrow$. It is well-known that $P$ is a complete lattice
if and only if $\eta_P$ is surjective. Since every finite lattice is a complete lattice, this
gives us a characterization of finite non-lattices, as follows.

\begin{fact}\label{fact:twomaxes}
A finite poset $P$ is a non-lattice if and only if there is a closed order ideal
$I$ with at least two maximal elements.
\end{fact}

\subsection{Ortholattices, orthomodular lattices}

An {\em ortholattice} is a bounded lattice $(L,\vee,\wedge,0,1,^\perp)$
equipped with an antitone mapping called {\em orthocomplementation}
$\perp\colon L\to L$ such that
\begin{itemize}
\item $0^\perp=1$, $1^\perp=0$
\item $x^\pp=x$
\item $(x\vee y)^\perp=x^\perp\wedge y^\perp$
\item $(x\wedge y)^\perp=x^\perp\vee y^\perp$
\item $x\wedge x^\perp=0$
\item $x\vee x^\perp=1$
\end{itemize}

An ortholattice is $L$ an {\em orthomodular lattice} if it satisfies the 
orthomodular law
$$
x\leq y\implies y=x\vee(y\wedge x^\perp),
$$
for all $x,y\in L$.

Similarly as for modular or distributive lattices, orthomodular lattices
can be characterized among ortholattices by absence of a ``forbidden sublattice''.
\begin{figure}
\begin{center}
\includegraphics{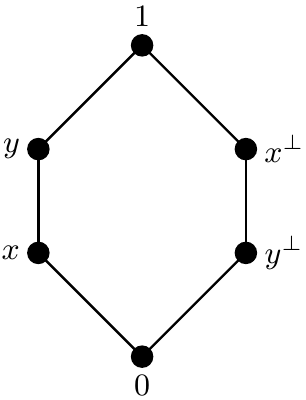}
\end{center}
\caption{The hexagon}
\label{fig:hexagon}
\end{figure}
\begin{proposition}(see for example \cite[Theorem 1.3.2]{kalmbach1977orthomodular})
An ortholattice $L$ is orthomodular if and only if it does not contain a sub-ortholattice isomorphic to {\em the hexagon} (see \Cref{fig:hexagon}).
\end{proposition}

A {\em Boolean algebra} is a distributive ortholattice.
We shall use the following well-known characterization.
\begin{proposition}\cite{cignoli1978deductive}
An ortholattice $L$ is a Boolean algebra if and only if,
for all $a,b\in L$,
\begin{equation}\label{eq:disjisortho}
a\wedge b=0\implies a\leq b^\perp.
\end{equation}
\end{proposition}
\begin{proof}
Clearly, every Boolean algebra satisfies \eqref{eq:disjisortho}.

By \cite[Proposition 1.3.7]{kalmbach1977orthomodular}
an ortholattice is a Boolean algebra iff it is uniquely complemented.
We shall prove that \eqref{eq:disjisortho} implies that 
$L$ is uniquely complemented.
Let $y$ be a complement of $x$. Since $x\wedge y=0$, $y\leq x^\perp$
by \eqref{eq:disjisortho}.
Since $x\vee y=1$, 
$$
0=1^\perp=(x\vee y)^\perp=x^\perp\wedge y^\perp.
$$
By \eqref{eq:disjisortho} $x^\perp\wedge y^\perp=0$ implies that $x^\perp\leq y^\pp=y$. Hence
$y=x^\perp$ and $y$ is unique.
\end{proof}

\section{Orthogonality spaces}

\begin{definition}
Let $V$ be a set, write $\pow(V)$ for the powerset of $V$. A mapping $\cl\colon
\pow(V)\to\pow(V)$ is a {\em closure operator} if it satisfies the conditions,
for all $X,Y\in\pow(V)$ 
\begin{enumerate}
\item $X\subseteq\cl(X)$,
\item $X\subseteq Y\implies \cl(X)\subseteq\cl(Y)$,
\item $\cl(\cl(X))=\cl(X)$.
\end{enumerate}
The pair $(V,\cl)$ is then called {\em a closure space}.
\end{definition}
In a closure space $(V,\cl)$, a set $X\subseteq V$ is called {\em closed} if
$\cl(X)=X$. It is easy to see that the set of all closed subsets of a closure space
forms a complete lattice under inclusion.

\begin{definition}
An {\em orthogonality space} $(O,\perp)$ is a set $O$ equipped with an irreflexive symmetric binary relation
$\perp\subseteq O\times O$.
\end{definition}

Let $(O,\perp)$ be an orthogonality space. 
We say that two elements $x,y\in O$ with $x\perp y$ are {\em
orthogonal}. A subset $B$ of $O$ is {\em pairwise orthogonal} if
every distinct pair of elements of $B$ is orthogonal.
Two subsets $X,Y$ of $O$ are {\em orthogonal} (in symbols $X\perp Y$) if 
$x\perp y$, for all $x\in X$ and $y\in Y$. We write $y\perp X$ for
$\{y\}\perp X$.
For every subset $X$ of $O$, we write 
\[
X^\perp=\{y\in O\mid y\perp X\}
\]
so that $X^\perp$ is the greatest subset of $O$ orthogonal to $X$.
Note that the mapping $X\mapsto X^\perp$ on the $(\pow(O),\subseteq)$ is antitone. Clearly, for every family $(X_i)_{i\in H}$ of subsets
of $O$,
$$
\bigl(\bigcup_{i\in H} X_i\bigr)^\perp=\bigcap_{i\in H}X_i^\perp.
$$
Since the orthocomplementation is antitone, the mapping $X\mapsto X^\pp$ is
isotone.
Moreover, for every $X\subseteq O$ we have 
$X\subseteq X^\pp$ and $X^\perp=X^{\pp\perp}$. Consequently, $X\mapsto X^\pp$
is a closure operator.  Note that $s\in X^\pp$ if and only if, for all $a\in
O$,
$$
a\perp X\implies s\perp a.
$$
We say that a subset $X$ of $O$ is {\em orthoclosed} if $X=X^\pp$.
The set of all orthoclosed subsets of an orthogonality space forms a complete ortholattice $L(O,\perp)$,
with meets given by intersection and joins given by $X\vee Y=(X^\perp\cap Y^\perp)^\perp$.
The smallest element of $L(O,\perp)$ is $\emptyset$ and the greatest element of
$L(O,\perp)$ is $O$. We say that $L(O,\perp)$ is {\em the logic} of $O$.

Let $(O,\perp)$ be an orthogonality space, let $X$ be an orthoclosed subset of
$O$. A maximal pairwise orthogonal subset of $X$ is called a {\em basis of $X$}.

In his PhD. thesis, J.C. Dacey proved the following theorem.
\begin{theorem}\label{thm:dacey}\cite{dacey1968orthomodular}
Let $(O,\perp)$ be an orthogonality space. Then $L(O,\perp)$ is an orthomodular lattice if
and only if for every orthoclosed subset $X$ of $O$ and every basis $B$ of $X$, $X=B^\pp$. 
\end{theorem}

An orthogonality space satisfying the condition of \Cref{thm:dacey} is called a {\em
Dacey space}. A orthoclosed subset $X$ in an orthogonality space such that $X=B^\pp$,
for every basis $B$ of $X$ is called a {\em Dacey set}.

\begin{figure}
\begin{center}
\vcenteredhbox{\includegraphics{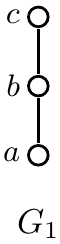}}
\hskip 1em
\vcenteredhbox{\includegraphics{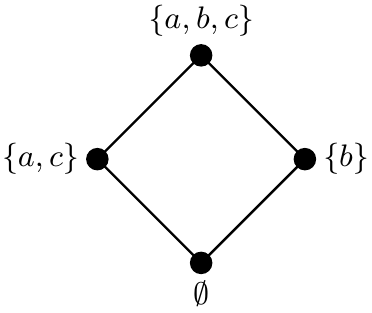}}
\hskip 2em
\vcenteredhbox{\includegraphics{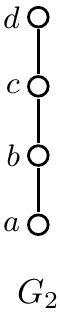}}
\hskip 1em
\vcenteredhbox{\includegraphics{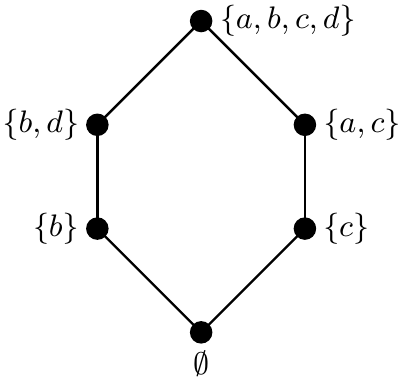}}
\caption{Two orthogonality spaces and their logics}
\label{fig:dacey}
\end{center}
\end{figure}

\begin{example}
Consider orthogonality spaces $G_1,G_2$ (\Cref{fig:dacey}). The logic
of $G_1$ is a 4-element Boolean algebra, hence an orthomodular lattice,
so $G_1$ is Dacey. The logic of $G_2$ is a hexagon, which is 
not an orthomodular lattice.
Therefore, by \Cref{thm:dacey}, $G_2$ is not Dacey. To see this directly,
consider the orthoclosed subset $\{b,d\}$. Then $\{b\}$ is a basis of
$\{b,d\}$.
However,
$$
\{b\}^\pp=\{a,c\}^\perp=\{b\}\neq\{b,c\},
$$
so $\{b,d\}$ is a non-Dacey orthoclosed set.
\end{example}
\begin{lemma}
\label{lemma:chardacey}
Let $X$ be an orthoclosed subset
an orthogonality space. The following are equivalent.
\begin{enumerate}[(a)]
\item $X$ is Dacey.
\item For every basis $B$ of $X$, $B^\perp=X^\perp$. 
\item For every basis $B$ of $X$, $B^\perp\subseteq X^\perp$.
\end{enumerate}
\end{lemma}
\begin{proof}~
\noindent\par(a)$\implies$(b): $B^\perp=B^{\pp\perp}=X^\perp$.
\noindent\par(b)$\implies$(c): Trivial.
\noindent\par(c)$\implies$(a): Since $B^\perp\subseteq X^\perp$, $X=X^\pp\subseteq B^\pp$. Since $B\subset X$, $B^\pp\subseteq X^\pp=X$. Since $X\subseteq B^\pp$ and
$B^\pp\subseteq X$, $B^\pp=X$ and $X$ is Dacey.
\end{proof}
\begin{lemma}
\label{lemma:unionofDaceys}
Let $X_1,X_2$ be two Dacey subsets of an orthogonality space $(O,\perp)$ with $X_1\perp X_2$. 
If $X_1\cup X_2$ is orthoclosed, then $X_1\cup X_2$ is Dacey.
\end{lemma}
\begin{proof}
Let $B$ be a basis of $X_1\cup X_2$. We claim that
$B\cap X_1$ is a basis of $X_1$. Since $B$ is pairwise orthogonal,
$B\cap X_1$ is pairwise orthogonal. Suppose that $B\cap X_1$ is
not a basis of $X_1$, and let $x\in(X_1\setminus B)$ be such
that $x\perp (B\cap X_1)$. Then $X_1\perp X_2$ implies that $x\perp X_2\supseteq B\cap X_2$ and thus
$$
x\perp (B\cap X_1)\cup(B\cap X_2)=B\cap(X_1\cup X_2)=B.
$$
This contradicts the maximality of $B$ in $X_1\cup X_2$, hence $B\cap X_1$
is a basis of $X_1$. Similarly, $B\cap X_2$ is a basis of $X_2$.

Since $X_1$ and $X_2$ are Dacey, we may use \Cref{lemma:chardacey} (b) to compute
\begin{multline*}
B^\perp=\bigl(B\cap(X_1\cup X_2)\bigr)^\perp=
\bigl((B\cap X_1)\cup(B\cap X_2)\bigr)^\perp=\\
(B\cap X_1)^\perp\cap(B\cap X_2)^\perp=X_1^\perp\cap X_2^\perp=(X_1\cup X_2)^\perp
\end{multline*}
and by \Cref{lemma:chardacey} (b), we see that $X_1\cup X_2$ is Dacey.
\end{proof}

\section{Orthogonality spaces associated with posets}

In this section, we will associate to every poset $P$ an orthospace. For a finite
bounded poset, we will prove that this orthospace is Dacey if and only if the
poset is a lattice.

For a poset $P$, we write $\Q(P)$ for the set of all pairs $(a,b)\in P\times P$
with $a<b$. 
In lattice theory, the elements of $\Q(A)$ are called {\em proper quotients}.

An element $(a,b)\in\Q(P)$ is denoted by $[a<b]$. A quotient $(a,b)$ is usually
denoted by $b/a$ but $[a<b]$ will be more handy for our purposes. For
$[a<b],[c<d]\in\Q(P)$, we write $[a<b]\leq[c<d]$ if $c\leq a<b\leq d$. Note
that $(\Q(P),\leq)$ is a subposet of the poset $P^*\times P$, where $P^*$ is
the dual poset of $P$ and that $(\Q(P),\leq)$ is isomorphic to the poset of
non-singleton closed intervals in $P$, ordered by inclusion.

For $[a<b],[c<d]\in\Q(P)$ we write $[a<b]\perp[c<d]$ if $b\leq c$ or $d\leq a$,
see \Cref{fig:orthogonality}.
Clearly, $\perp$ is symmetric and irreflexive, so $(\Q(P),\perp)$ is an
orthogonality space. Note that $\emptyset$ is orthoclosed in $(\Q(P),\perp)$.

\begin{figure}[t]
\begin{center}
\includegraphics{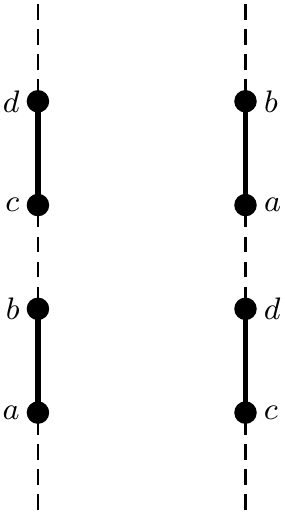}
\end{center}
\caption{The orthogonality relation in $\Q(P)$}
\label{fig:orthogonality}
\end{figure}

We equipped $\Q(P)$ with two relations: the orthogonality $\perp$ and the partial order
$\leq$. Various classes of posets equipped with an orthogonality relation were considered in
\cite{cattaneo1974abstract}. In the terminology of that paper, every $(\Q(P),\perp,\leq)$
is a {\em poset with a strong degenerate orthogonality}.

\begin{lemma}\label{lemma:orthogonaliscomparable}
Let $P$ be a poset, let 
$$
B=\{[a_1<a_2],\dots,[a_{2n-1}<a_{2n}]\}
$$ be a finite pairwise orthogonal subset of $\Q(P)$.
Then $\{a_1,a_2,\dots,a_{2n}\}$ are pairwise comparable.
\end{lemma}
\begin{proof}
Note that whenever $[a<b]\perp[c<d]$, the elements in the set $\{a,b,c,d\}$ are
pairwise comparable. The rest follows by an easy induction.
\end{proof}

Observe that $[u<v]\leq [a<b]\perp[c<d]$ implies that $[u<v]\perp[c<d]$.

\begin{lemma}
\label{lemma:orthoclosedisdownset}
Let $P$ be a poset. Every orthoclosed subset of $(\Q(P),\perp)$ is a lower set in
$(\Q(P),\leq)$.
\end{lemma}
\begin{proof}
Let $X$ be an orthoclosed subset of $(\Q(P),\perp)$.
Let $[a<b]\in X$ and $[u<v]\leq[a<b]$. We need to prove that $[u<v]\in X=X^\pp$,
that means, for all $[c<d]\in\Q(P)$, $[c<d]\perp X$ implies $[c<d]\perp[u<v]$. However,
this is easy: if $[c<d]\perp X$, then $[c<d]\perp[a<b]$ and since $[u<v]\leq[a<b]$,
$[c<d]\perp[u<v]$.
\end{proof}

Let $U$ be a subset of a poset $P$. We denote
\begin{align*}
\tau(U)=\{[a<b]\in\Q(P)\mid b\in U\}\\
\beta(U)=\{[a<b]\in\Q(P)\mid a\in U\}.
\end{align*}
Note that, for all $[x<y]\in\Q(P)$,
\begin{align*}
[x<y]^\perp&=\tau(x^\downarrow)\cup\beta(y^\uparrow)\\
[x<y]^\downarrow&=\beta(x^\uparrow)\cap\tau(y^\downarrow).
\end{align*}

\begin{lemma}
\label{lemma:derp}~
\begin{enumerate}[(a)]
\item For every lower set $I$ of a lower-bounded poset $P$, $\tau(I)^\perp=\beta(I^\uparrow)$.
\item For every upper set $F$ of a upper-bounded poset $P$, $\beta(F)^\perp=\tau(F^\downarrow)$.
\end{enumerate}
\end{lemma}
\begin{proof}

\begin{enumerate}[(a)]
\item Let us prove that $\tau(I)^\perp\subseteq\beta(I^\uparrow)$.
Let $[x<y]\in\tau(I)^\perp$. We need to prove that 
$[x<y]\in\beta(I^\uparrow)$, that means that, for all $b\in I$, $b\leq x$. If $b=0$, there is nothing to prove,
so let us assume that $0<b$. Then $[0<b]\in\tau(I)$, so $[0<b]\perp[x<y]$ and this clearly
implies $b\leq x$.

For the opposite inclusion, let $[x<y]\in\beta(I^\uparrow)$, meaning that $x\in
I^\uparrow$. For all $[a<b]\in\tau(I)$, $b\in I$ and thus $b\leq x$. Therefore,
$[a<b]\perp[x<y]$ and hence $[x<y]\in\tau(I)^\perp$.

\item Dually.
\end{enumerate}
\end{proof}
\begin{corollary}
\label{coro:tauclosure}
For every lower set $I$ of a bounded poset $P$,
$\tau(I)^\pp = \tau(I^\ud)$
\end{corollary}
\begin{proof}
Using (a) and (b) of \Cref{lemma:derp},
$$
\tau(I)^\pp=\beta(I^\uparrow)^\perp=\tau(I^\ud)
$$
\end{proof}

\begin{corollary}
\label{coro:doubleperpisdown}
Let $P$ be a poset.
For every $[x<y]\in\Q(P)$, 
$$
[x<y]^\pp=[x<y]^\downarrow
$$
\end{corollary}
\begin{proof}
By \Cref{coro:tauclosure},
\begin{align*}
[x<y]^\pp=&(\tau(x^\downarrow)\cup\beta(y^\uparrow))^\perp=
(\tau(x^\downarrow)^\perp)\cap(\beta(y^\uparrow)^\perp)=\\
&(\beta(x^\du)\cap\tau(y^\ud))=
\beta(x^\uparrow)\cap\tau(y^\downarrow)=[x<y]^\downarrow
\end{align*}
\end{proof}
\begin{lemma}
\label{lemma:downisdacey}
Let $P$ be a poset.
For every $[x<y]\in\Q(P)$, $[x<y]^{\downarrow}$ is Dacey.
\end{lemma}
\begin{proof}
By \Cref{coro:doubleperpisdown}, $[x<y]^{\downarrow}$ is orthoclosed.

Let $B$ be a basis of $[x<y]^{\downarrow}$.
Since $B$ is pairwise orthogonal, we may write
$$
B=\{[c_1<d_1],\dots,[c_k<d_k]\}
$$
and by \Cref{lemma:orthogonaliscomparable} we may assume without loss of generality that $d_i\leq c_{i+1}$, for
$i\in\{1,\dots,k-1\}$.  If $x<c_1$, then $[x<c_1]\perp B$ and $B$ is not
maximal pairwise orthogonal, so $x=c_1$. Similarly, $d_k=y$ and, for all
$i\in\{1,\dots,k-1\}$, $d_i=c_{i+1}$.  It is now obvious that $[p<q]\in
B^\perp$ implies either $q\leq c_1=x$ or $y=d_k\leq p$, so $[p<q]\in
[x<y]^{\perp}=([x<y]^\downarrow)^\perp$.  Therefore, $B^\perp\subseteq
([x<y]^\downarrow)^\perp$ and by \Cref{lemma:chardacey} (c), this implies that
$[x<y]^\downarrow$ is Dacey.

\end{proof}

\begin{lemma}
\label{lemma:emma}
Let $P$ be a lattice and let $X$ be an orthoclosed subset
of $\Q(P)$. Let $[a<b],[c<d]\in X$ be such that 
$[a<b]\not\perp[c<d]$. Then
\begin{enumerate}[(a)]
\item $\{[a<b],[c<d]\}^\perp=\beta((b\vee d)^\uparrow)\cup\tau((a\wedge
c)^\downarrow)$
\item $[a\wedge c<b \vee d]\in X$
\end{enumerate}
\end{lemma}
\begin{proof}~

\begin{enumerate}[(a)]
\item
The inclusion $\{[a<b],[c<d]\}^\perp\supseteq\beta((b\vee d)^\uparrow)\cup\tau((a\wedge
c)^\downarrow)$ is clear, let us prove the other one.
Let $[x<y]\in\{[a<b],[c<d]\}^\perp$.
Since $[x<y]\perp[a<b]$, either $b\leq x$ or $y\leq a$.

Suppose that $b\leq x$. Since $[x<y]\perp[c<d]$, either $y\leq c$ or
$d\leq x$. If $y\leq c$ then
$a<b\leq x<y\leq c<d$, which contradicts $[a<b]\not\perp[c<d]$. Therefore, $d\leq
x$, so $b\vee d\leq x$ meaning that $[x<y]\in\beta((b\vee d)^\uparrow)$.

Similarly $y\leq a$ implies that $y\leq a\wedge c$, hence
$[x<y]\in\tau((a\wedge c)^\downarrow)$.

\item 
Clearly, 
$$
\beta((b\vee d)^\uparrow)\cup\tau((a\wedge c)^\downarrow)=[a\wedge c<b\vee d]
^\perp.
$$
Therefore by (a), $[x<y]\in\{[a<b],[c<d]\}^\perp$ implies that 
$[x<y]\perp[a\wedge c<b\vee d]$, so
$$
[a\wedge c<b\vee d]\in\{[a<b],[c<d]\}^\pp\subseteq X^\pp=X.
$$
\end{enumerate}
\end{proof}
\begin{lemma}
\label{lemma:mergetouching}
Let $P$ be a poset, let $X$ be an orthoclosed subset of $\Q(P)$. If $[p<q_2],[q_1<r]\in X$
and $q_1\leq q_2$ and $p<r$, then $[p<r]\in X$.
\end{lemma}
\begin{proof}
Let $[x<y]\in\Q(P)$ be such that $[x<y]\in\{[p<q_2],[q_1<r]\}^\perp$. We shall prove that
this implies $y\leq p$ or $r\leq x$. Indeed, assume that $y\not\leq p$ and $r\not\leq x$.
Then $[x<y]\perp[p<q_2]$ implies $q_2\leq x$ and $[x<y]\perp[q_1<r]$ implies $y\leq q_1$.
Hence $q_2\leq x<y\leq q_1$ and this contradicts $q_1\leq q_2$.

We have proved that $y\leq p$ or $r\leq x$ and 
each of these implies $[x<y]\perp[p<r]$. Therefore,
$$
[p<r]\in\{[p<q_2],[q_1<r]\}^\pp\subseteq X^\pp=X
$$
\end{proof}

Let $P$ be a finite poset. We say that an lower subset $X$ of $(\Q(P),\leq)$ is {\em of
chain type} if for all pairs $[a<b]\neq[c<d]$ of maximal elements of $X$ of $\Q(P)$, we
have $[a<b]\perp[c<d]$ and either $b<c$ or $d<a$.

\begin{lemma}\label{lemma:chaintypeisdacey} Let $P$ be a finite poset.
Every $X\subset \Q(P)$ of a chain type is Dacey.
\end{lemma}
\begin{proof}
Let us prove that $X$ is orthoclosed.
Write
$$
\max(X)=\{[a_1<a_2],\dots,[a_{2n-1}<a_{2n}]\}.
$$
Without loss of generality, assume that $a_1<a_2<a_3<\dots<a_{2n-1}<a_{2n}$.
Obviously,
$$
X=[a_1<a_2]^\downarrow\cup\dots\cup[a_{2n-1}<a_{2n}]^\downarrow
$$
Note that
$$
\max(X)^\perp=\tau(a_1^\downarrow)\cup[a_2<a_3]^\downarrow\cup
[a_4<a_5]^\downarrow\cup\dots\cup[a_{2n-2}<a_{2n-1}]^\downarrow\cup\beta(a_{2n}^\uparrow).
$$
It is easy to see that  $X=\max(X)^\pp$.
In particular, $X$ is orthoclosed. 

To prove that $X$ is Dacey, we shall use induction with respect to $n$.

If $n=0$, then $X=\emptyset$ and $X$ is Dacey.

Let $n>0$.

Put $X_1=[a_{2n-1}<a_{2n}]^\downarrow$,
$X_2=X\setminus([a_{2n-1}<a_{2n}]^\downarrow)$ and note that $X_1\perp X_2$. 
By \Cref{lemma:downisdacey}, $[a_{2n-1}<a_{2n}]^{\downarrow}$ is Dacey.
The set $X_2$ is of a chain type, hence $X_2$ is Dacey by induction hypothesis.
The
conditions of \Cref{lemma:unionofDaceys} are satisfied, hence $X=X_1\cup X_2$
is Dacey.

\end{proof}

\begin{theorem}
\label{thm:main}
Let $P$ be a finite bounded poset.
The following are equivalent.
\begin{enumerate}[(a)]
\item $P$ is a lattice.
\item Every orthoclosed subset of $\Q(P)$ is of chain type.
\item $\Q(P)$ is Dacey.
\item $L(\Q(P))$ is an orthomodular lattice.
\end{enumerate}
\end{theorem}
\begin{proof}~
\par\noindent(a)$\implies$(b): Let $X$ be an orthoclosed subset of $\Q(P)$,
let $[a<b],[c<d]\in\max(X)$ with $[a<b]\neq[c<d]$.
Let us prove that $[a<b]\perp[c<d]$.
Assume the contrary; by \Cref{lemma:emma} (b), $[a\wedge c<b\vee d]\in X$. As
$[a<b]\leq[a\wedge c<b\vee d]$ and $[a<b]$ is maximal in $X$,
$[a<b]=[a\wedge c<b\vee d]$. This implies $a=a\wedge c$ and
$b=b\vee d$, in other words, $a\leq c$ and $d\leq b$ so
$[c<d]\leq[a<b]$, hence $[c<d]=[a<b]$ because $[c<d]$ is maximal in $\Q(P)$
and we have reached the desired contradiction.

From $[a<b]\perp[c<d]$ it follows that either $b\leq c$ or $d\leq a$. Assume that
$b=c$. Then, by \Cref{lemma:mergetouching}, $[a<d]\in X$ and this contradicts
the maximality of $[a<b]$ and $[b<c]$ in $X$.

\par\noindent(b)$\implies$(c):
By \Cref{lemma:chaintypeisdacey}.
\par\noindent(c)$\implies$(a):
Suppose that $P$ is a non-lattice. We need to prove that $\Q(P)$ is
not Dacey. By \Cref{fact:twomaxes}, since $P$ is a non-lattice, 
there exists a lower set $I$ of $P$ with $I=I^\ud$ such
that $I$ has at least two maximal elements.
Let $a,b$ be two distinct maximal elements of $I$. By \Cref{coro:tauclosure},
$$
\tau(I)^\pp=\tau(I^\ud)=\tau(I),
$$
hence $\tau(I)$ is orthoclosed. Clearly, $\{[0<a]\}$ is a basis of $\tau(I)$
and we may apply \Cref{lemma:derp} to compute
$$
\{[0<a]\}^\pp=(\beta(a^\uparrow))^\perp=\tau(a^\ud)=\tau(a^\downarrow).
$$
However, as $[0<b]\notin\tau(a^\downarrow)$, $\tau(a^\downarrow)\neq I$. Therefore,
$(\Q(P),\perp)$ is not Dacey.

\par\noindent(c)$\Leftrightarrow$(d): By \Cref{thm:dacey}.

\end{proof}
\begin{figure}
\begin{center}
\vcenteredhbox{\includegraphics{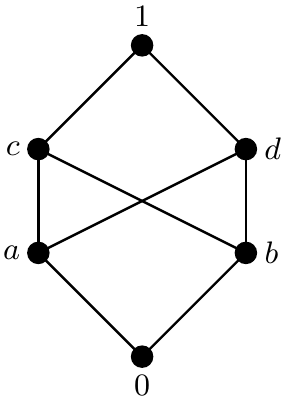}}
\hskip 3em
\vcenteredhbox{\includegraphics{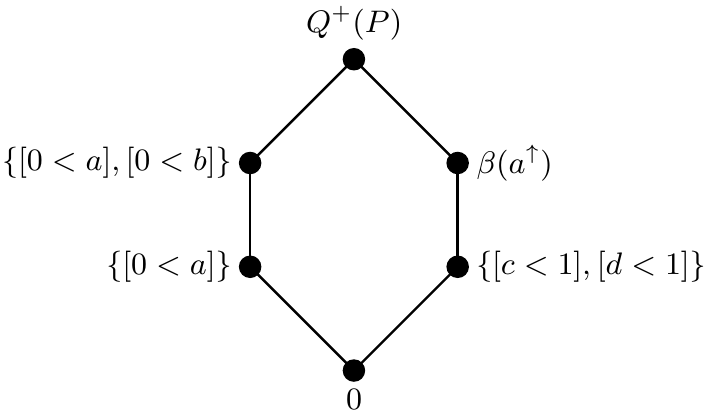}}
\caption{A non-lattice poset and a sub-ortholattice of its logic}
\label{fig:r6}
\end{center}
\end{figure}

\begin{example}
Consider the bounded poset $P$ on the left side of \Cref{fig:r6}.
Since it is a non-lattice, \Cref{thm:main} tells us that $(\Q(P),\perp)$ is not a Dacey space or, equivalently, that  
$L(\Q(P),\perp)$ is not orthomodular. Hence
$L(\Q(P),\perp)$ must contain a hexagon as its sub-ortholattice. One of such 
hexagons is depicted on the right side of \Cref{fig:r6}.
To see more directly that $(\Q(P),\perp)$ is not a Dacey space, just consider the 
orthoclosed set $\{[0<a],[0<b]\}$ and its basis $\{[0<a]\}$.
\end{example}

\begin{theorem}\label{thm:chainboolean}
Let $P$ be a bounded poset.
Then $P$ is a chain if and only if $L(\Q(P))$ is a Boolean algebra.
\end{theorem}
\begin{proof}~
Suppose that $P$ is a chain.
$L(\Q(P))$ is a Boolean algebra
if and only if, for all orthoclosed subsets $X,Y$,
$$
X\cap Y=\emptyset\implies X\perp Y
$$
Assume the contrary, that means, $X\cap Y=\emptyset$ and $X,Y$ are not
orthogonal.
There are $[x_1<x_2]\in X$ and
$[y_1<y_2]\in Y$ such that $[x_1<x_2]\not\perp[y_1<y_2]$.
As $[x_1<x_2]\not\perp[y_1<y_2]$, we see that $x_2\not\leq y_1$ and
$y_2\not\leq x_1$. Since $P$ is a chain, this is equivalent to
$y_1<x_2$ and $x_1<y_2$.
\begin{description}
\item[(Case 1)]$x_1\leq y_1$: we see that
$x_1\leq y_1<x_2$, hence $[y_1<x_2]\leq[x_1\leq x_2]$ and
$[y_1<x_2]\in X$.
\item[(Case 1a)]$y_2\leq x_2$: this implies $x_1\leq y_1<y_2\leq x_2$,
hence $[y_1<y_2]\in X$ and $X,Y$ are not disjoint.
\item[(Case 1b)]$x_2<y_2$: we see that $x_1\leq y_1<x_2\leq y_2$, hence
$[y_1<x_2]\leq[y_1<y_2]\in Y$ and $[y_1<x_2]\in Y$ and $X,Y$ are not disjoint.
\item[(Case 2)]$y_1<x_1$: we see that $y_1<x_1<x_2$.
\item[(Case 2a)]$x_2\leq y_2$: this implies $y_1<x_1<x_2\leq y_2$, hence
$[x_1<x_2]\leq[y_1<y_2]$ and $[x_1<x_2]\in Y$ and $X,Y$ are not disjoint.
\item[(Case 2b)]$y_2<x_2$: this implies $y_1<x_1<y_2<x_2$, hence
$[x_1<y_2]\leq[x_1<x_2]\in X$ and $[x_1<y_2]\leq[y_1<y_2]\in Y$ and
$X,Y$ are not disjoint.
\end{description}

Suppose that $P$ is not a chain. Let $a,b$ be incomparable elements of $P$.
Then $[a<1]^\downarrow$ and $[0<b]^\downarrow$ are disjoint elements of $L(\Q(P))$
that are not orthogonal. Hence $L(\Q(P))$ is not a Boolean algebra.
\end{proof}

Let us remark that \Cref{thm:chainboolean} is not true if we do not assume that $P$ is bounded. To see this, let $P$ be a two-element antichain. Then $\Q(P)$ is empty and $L(\Q(P),\perp)$ is a two-element Boolean algebra.

\section{Connections to other constructions}

Let us extend our $[x<y]$ notation to an arbitrary finite number of elements:
if $C=\{x_1,\dots,x_n\}$ is a finite chain in a poset $P$,
we write $C=[x_1<\dots<x_n]$ to indicate the partial order.

The following construction was discovered by Kalmbach in \cite{kalmbach1977orthomodular}.
Let $P$ be a bounded lattice, write
$$
K(P)=\{C:C\text{ is a finite chain in $P$ with even number of elements}\}
$$
Define a partial order on $K(P)$ by the following rule:
$$
[x_1<x_2<\dots<x_{2n-1}<x_{2n}]\leq[y_1<y_2<\dots<y_{2k-1}<y_{2k}]
$$
if for every $1\leq i\leq n$ there is $1\leq j\leq k$ such that
$$
y_{2j-1}\leq x_{2i-1}<x_{2i}\leq y_{2j}.
$$
Define a unary operation $C\mapsto C^\perp$ on $K(P)$ to be the symmetric
difference with the set $\{0,1\}$. 

\begin{theorem}\cite{kalmbach1977orthomodular}
For every bounded lattice $P$, $K(P)$ is an orthomodular lattice.
\end{theorem}

The lattice $K(P)$ is called {\em the Kalmbach construction}.

\begin{theorem}\label{thm:logickalmbach}
For every finite lattice $P$, $K(P)\simeq L(\Q(P),\perp)$.
\end{theorem}
\begin{proof}
For $C=[x_1<x_2<\dots<x_{2n-1}<x_{2n}]\in K(P)$, write
$f(C)=[x_1<x_2]^\downarrow\cup\dots[x_{2n-1}<x_{2n}]^\downarrow$.
Then $f(C)$ is a lower subset of $\Q(P)$ the chain type. By
\Cref{lemma:chaintypeisdacey}, $f(C)$ is Dacey and hence orthoclosed.
Obviously, $f$ is isotone, $f(\emptyset)=\emptyset$ and
$f(P)=f([0<1])=[0<1]^\downarrow=\Q(P)$, hence $f\colon K(P)\to L(\Q(P),\perp)$
is an isotone mapping preserving the bounds. Moreover,
$f$ preserves the complementation. There
are four cases to check, each of the cases consisting of
a conjunction of one of $(0\in C,0\notin C)$ and one of
$(1\in C,1\notin C)$. We will check one of the cases,
the remaining three are very similar.

Suppose that $0\in C$, $1\notin C$, that means, $0=x_1$ and $x_{2n}<1$:
$$
f(C)=f([0<x_2<\dots<x_{2n-1}<x_{2n}])=[0<x_2]^\downarrow\cup\dots\cup[x_{2n-1}<x_{2n}]^\downarrow
$$
Then
$$
f(C^\perp)=f([x_2<x_3<\dots<x_{2n}<1])=[x_2<x_3]^\downarrow\cup\dots\cup[x_{2n}<1]^\downarrow
$$
It is easy to see that $f(C)^\perp=f(C^\perp)$.

Let us construct an isotone map inverse to $f$. Let $X$ be an orthoclosed
subset of $\Q(P)$. By \Cref{thm:main}, $X$ is of the chain type. Write
$\max(X)=\{[x_1<x_2],\dots,[x_{2n-1}<x_{2n}]\}$ and assume,
without loss of generality, that
$x_1<x_2<\dots<x_{2n-1}<x_{2n}$. This chain is an element of $K(P)$,
put $g(X)=[x_1<x_2<\dots<x_{2n-1}<x_{2n}]$. It is easy to check
$g$ is isotone, preserves the bounds and preserves the orthocomplementation.
\end{proof}

For general infinite lattices, \Cref{thm:logickalmbach} is not true. For example,
for the real unit interval $[0,1]_{\mathbb R}$, the Kalmbach construction 
$K([0,1]_{\mathbb R})$ is not a complete lattice, whereas every $L(\Q(P),\perp)$ is
a complete lattice.

For every poset $P$, its Dedekind-MacNeille completion $\compl P$ can be
represented as the complete lattice of all lower sets $I$ with $I=I^\ud$. For
a family $(I_j)_{j\in H}$ of elements of $\compl P$, their meet is their intersection
and their join is the closure of their union:
$$
\bigvee_{j\in H}I_j=(\bigcup_{j\in H}I_j)^\ud
$$

By \Cref{coro:tauclosure}, for every lower set $I$ we have $\tau(I)^\pp=\tau(I^\ud)$.
In particular, $\tau(I^\ud)$ is orthoclosed. Therefore, for every $I\in\compl P$,
$\tau(I)$ is orthoclosed.
\begin{theorem}
For every bounded poset $P$, $\tau\colon\compl P\to L(\Q(P),\perp)$ is an injective
morphism of complete lattices.
\end{theorem}
\begin{proof}
Clearly, $\tau$ is injective and preserves the bounds of $\compl P$.
Let $(I_j)_{j\in H}$ be a family of elements of $\compl P$. Clearly,
$$
\tau(\bigcap_{j\in H}I_j)=\bigcap_{j\in H}\tau(I_j)
$$
so $\tau$ preserves all meets. For joins, we may compute
\begin{multline*}
\bigvee_{j\in H}\tau(I_j)=\bigl(\bigcup_{j\in H}\tau(I_j)\bigr)^\pp=\bigl(\bigcap_{j\in H}\tau(I_j)^\perp\bigr)^\perp=
\bigl(\bigcap_{j\in H}\beta(I_j^\uparrow)\bigr)^\perp=\\
\beta\bigl((\bigcup_{j\in H}I_j)^\uparrow\bigr)^\perp=
\tau\bigl((\bigcup_{j\in H}I_j)^\ud\bigr)=\tau(\bigvee_{j\in H}I_j)
\end{multline*}
\end{proof}

Let us close with a conjecture.
\begin{conjecture}
There exists an infinite bounded non-lattice $P$ such that $(\Q(P),\perp)$ is
Dacey.
\end{conjecture}

\section*{Declarations}
\begin{description}
\item[\bf Funding:]
This research is supported by grants VEGA 2/0142/20 and 1/0006/19,
Slovakia and by the Slovak Research and Development Agency under the contracts
APVV-18-0052 and APVV-20-0069.
\item[\bf Other support:]The author expresses his gratitude for the possibility to present this material 
on the conference SSAOS 2022 in Tatranská Lomnica.
\item[\bf Conflict of interest/Competing interest:]None.
\item[\bf Availability of data and material:]Not applicable.
\item[\bf Code availability:]Not applicable.
\item[\bf Authors' contributions:]There is a single author. 
\end{description}


\end{document}